\documentclass[10pt,reqno]{amsart}
\usepackage{latexsym,amsmath,amssymb,amscd}
\usepackage[all]{xy}
\usepackage{enumerate}
\usepackage{hyperref}

\makeatletter
  \newcommand\@dotsep{4}
  \def\@tocline#1#2#3#4#5#6#7{\relax
     \ifnum #1>\c@tocdepth 
     \else
     \par \addpenalty\@secpenalty\addvspace{#2}%
     \begingroup \hyphenpenalty\@M
     \@ifempty{#4}{%
     \@tempdima\csname r@tocindent\number#1\endcsname\relax
        }{%
         \@tempdima#4\relax
           }%
      \parindent\z@ \leftskip#3\relax \advance\leftskip\@tempdima\relax
      \rightskip\@pnumwidth plus1em \parfillskip-\@pnumwidth
       #5\leavevmode\hskip-\@tempdima #6\relax
       \leaders\hbox{$\m@th
       \mkern \@dotsep mu\hbox{.}\mkern \@dotsep mu$}\hfill
       \hbox to\@pnumwidth{\@tocpagenum{#7}}\par
       \nobreak
        \endgroup
         \fi}
\makeatother 

\begin{document}

\makeatletter
\@addtoreset{figure}{section}
\def\thefigure{\thesection.\@arabic\c@figure}
\def\fps@figure{h,t}
\@addtoreset{table}{bsection}

\def\thetable{\thesection.\@arabic\c@table}
\def\fps@table{h, t}
\@addtoreset{equation}{section}
\def\theequation{
\arabic{equation}}
\makeatother

\newcommand{\bfi}{\bfseries\itshape}

\newtheorem{theorem}{Theorem}
\newtheorem{acknowledgment}[theorem]{Acknowledgment}
\newtheorem{corollary}[theorem]{Corollary}
\newtheorem{definition}[theorem]{Definition}
\newtheorem{example}[theorem]{Example}
\newtheorem{lemma}[theorem]{Lemma}
\newtheorem{notation}[theorem]{Notation}
\newtheorem{problem}[theorem]{Problem}
\newtheorem{proposition}[theorem]{Proposition}
\newtheorem{remark}[theorem]{Remark}
\newtheorem{setting}[theorem]{Setting}
\newtheorem{hypothesis}[theorem]{Hypothesis}

\numberwithin{theorem}{section}
\numberwithin{equation}{section}

\renewcommand{\1}{{\bf 1}}
\newcommand{\Ad}{{\rm Ad}}
\newcommand{\Alg}{{\rm Alg}\,}
\newcommand{\Aut}{{\rm Aut}\,}
\newcommand{\ad}{{\rm ad}}
\newcommand{\Borel}{{\rm Borel}}
\newcommand{\botimes}{\bar{\otimes}}
\newcommand{\Ci}{{\mathcal C}^\infty}
\newcommand{\Cint}{{\mathcal C}^\infty_{\rm int}}
\newcommand{\Conv}{{\rm Conv}}
\newcommand{\Cpol}{{\mathcal C}^\infty_{\rm pol}}
\newcommand{\comp}{{\rm comp}}
\newcommand{\Der}{{\rm Der}\,}
\newcommand{\Diff}{{\rm Diff}\,}
\newcommand{\de}{{\rm d}}
\newcommand{\ee}{{\rm e}}
\newcommand{\End}{{\rm End}\,}
\newcommand{\ev}{{\rm ev}}
\newcommand{\hotimes}{\widehat{\otimes}}
\newcommand{\id}{{\rm id}}
\newcommand{\ie}{{\rm i}}
\newcommand{\iotaR}{\iota^{\rm R}}
\newcommand{\GL}{{\rm GL}}
\newcommand{\gl}{{{\mathfrak g}{\mathfrak l}}}
\newcommand{\Hom}{{\rm Hom}\,}
\newcommand{\Img}{{\rm Im}\,}
\newcommand{\Ind}{{\rm Ind}}
\newcommand{\Ker}{{\rm Ker}\,}
\newcommand{\Lie}{\text{\bf L}}
\newcommand{\Mt}{{{\mathcal M}_{\text t}}}
\newcommand{\m}{\text{\bf m}}
\newcommand{\pr}{{\rm pr}}
\newcommand{\Ran}{{\rm Ran}\,}
\renewcommand{\Re}{{\rm Re}\,}
\newcommand{\so}{\text{so}}
\newcommand{\spa}{{\rm span}\,}
\newcommand{\supp}{{\rm supp}\,}
\newcommand{\Tr}{{\rm Tr}\,}
\newcommand{\tw}{\ast_{\rm tw}}
\newcommand{\U}{{\rm U}}
\newcommand{\UCb}{{{\mathcal U}{\mathcal C}_b}}
\newcommand{\weak}{\text{weak}}

\newcommand{\CC}{{\mathbb C}}
\newcommand{\RR}{{\mathbb R}}
\newcommand{\TT}{{\mathbb T}}
\newcommand{\NN}{{\mathbb N}}

\newcommand{\Ac}{{\mathcal A}}
\newcommand{\Bc}{{\mathcal B}}
\newcommand{\Cc}{{\mathcal C}}
\newcommand{\Dc}{{\mathcal D}}
\newcommand{\Ec}{{\mathcal E}}
\newcommand{\Fc}{{\mathcal F}}
\newcommand{\Hc}{{\mathcal H}}
\newcommand{\Jc}{{\mathcal J}}
\newcommand{\Kc}{{\mathcal K}}
\newcommand{\Lc}{{\mathcal L}}
\renewcommand{\Mc}{{\mathcal M}}
\newcommand{\Nc}{{\mathcal N}}
\newcommand{\Oc}{{\mathcal O}}
\newcommand{\Pc}{{\mathcal P}}
\newcommand{\Qc}{{\mathcal Q}}
\newcommand{\Sc}{{\mathcal S}}
\newcommand{\Tc}{{\mathcal T}}
\newcommand{\Vc}{{\mathcal V}}
\newcommand{\Uc}{{\mathcal U}}
\newcommand{\Xc}{{\mathcal X}}
\newcommand{\Yc}{{\mathcal Y}}
\newcommand{\Zc}{{\mathcal Z}}

\newcommand{\Bg}{{\mathfrak B}}
\newcommand{\Fg}{{\mathfrak F}}
\newcommand{\Gg}{{\mathfrak G}}
\newcommand{\Ig}{{\mathfrak I}}
\newcommand{\Jg}{{\mathfrak J}}
\newcommand{\Lg}{{\mathfrak L}}
\newcommand{\Pg}{{\mathfrak P}}
\newcommand{\Sg}{{\mathfrak S}}
\newcommand{\Xg}{{\mathfrak X}}
\newcommand{\Yg}{{\mathfrak Y}}
\newcommand{\Zg}{{\mathfrak Z}}

\newcommand{\ag}{{\mathfrak a}}
\newcommand{\bg}{{\mathfrak b}}
\newcommand{\dg}{{\mathfrak d}}
\renewcommand{\gg}{{\mathfrak g}}
\newcommand{\hg}{{\mathfrak h}}
\newcommand{\kg}{{\mathfrak k}}
\newcommand{\mg}{{\mathfrak m}}
\newcommand{\n}{{\mathfrak n}}
\newcommand{\og}{{\mathfrak o}}
\newcommand{\pg}{{\mathfrak p}}
\newcommand{\qg}{{\mathfrak q}}
\newcommand{\sg}{{\mathfrak s}}
\newcommand{\tg}{{\mathfrak t}}
\newcommand{\ug}{{\mathfrak u}}
\newcommand{\zg}{{\mathfrak z}}

\newcommand{\ZZ}{\mathbb Z}
\newcommand{\BB}{\mathbb B}
\newcommand{\HH}{\mathbb H}

\newcommand{\ep}{\varepsilon}

\newcommand{\hake}[1]{\langle #1 \rangle }

\newcommand{\scalar}[2]{(#1 \mid#2) }
\newcommand{\dual}[2]{\langle #1, #2\rangle}

\newcommand{\norm}[1]{\Vert #1 \Vert }

\newcommand{\blue}[1]{\textcolor{blue}{#1}}
\newcommand{\red}[1]{\textcolor{red}{#1}}
\newcommand{\green}[1]{\textcolor{green}{#1}}

\markboth{}{}

\makeatletter
\title[Banach space representations of nilpotent Lie groups. Part 1]{Transference for Banach space representations of nilpotent Lie groups.\\ Part 1. Irreducible representations}
\author{Ingrid Belti\c t\u a, Daniel Belti\c t\u a, and Jos\'e E. Gal\'e}
\address{Institute of Mathematics ``Simion Stoilow'' 
of the Romanian Academy, 
P.O. Box 1-764, Bucharest, Romania}
\email{Ingrid.Beltita@imar.ro}
\address{Institute of Mathematics ``Simion Stoilow'' 
of the Romanian Academy, 
P.O. Box 1-764, Bucharest, Romania}
\email{Daniel.Beltita@imar.ro}
\address{Departamento de matem\'aticas and I.U.M.A.,
Universidad de Zaragoza, 50009 Zaragoza, Spain}
\email{gale@unizar.es}
\thanks{This research was partly  supported by Project MTM2013-42105-P and Project MTM2016-77710-P, fondos FEDER, Spain. 
The two first-named authors have also been supported by a Grant of the Romanian National Authority for Scientific Research, CNCS-UEFISCDI, project number PN-II-ID-PCE-2011-3-0131. 
The third-named author has also been supported by Project E-64, D.G. Arag\'on, Spain.}

\keywords{reflexive Banach space; nilpotent Lie group; transference}
\subjclass[2000]{Primary 17B30; Secondary 22E25, 22E27}
\makeatother

\begin{abstract} 
We establish a general CCR (liminarity) property for uniformly bounded irreducible representations of 
nilpotent Lie groups on reflexive Banach spaces, 
extending the well known property of unitary irreducible representations of these groups on Hilbert spaces. 
We also prove that this conclusion fails for many representations on non-reflexive Banach spaces. 
Our approach to these results blends the method of transference from abstract harmonic analysis 
and a systematic use of spaces of smooth vectors with respect to Lie group representations. 
\end{abstract}

\maketitle


\section{Introduction}\label{introduction}


This paper was partially motivated by the well-known fact in operator theory that if $\Hc$ is a separable complex Hilbert space and $\Ac\subseteq\Bc(\Hc)$ is an associative $*$-subalgebra for which there are no nontrivial invariant subspaces, with its norm-closure $\overline{\Ac}$, 
then the following implication holds true: 
$$\Kc(\Hc)\cap\Ac\ne\{0\}\implies\Kc(\Hc)\subseteq\overline{\Ac},  $$
where $\Kc(\Hc)$ is the set of all compact operators on~$\Hc$. 
In particular, if 
$\Ac\subseteq\Kc(\Hc)$ then $\overline{\Ac}=\Kc(\Hc)$.
These implications do not carry over directly to operator algebras on Banach spaces (cf. \cite{RaRo00}), 
but we explore this phenomenon from the perspective of representation theory of Lie groups, as follows. 

If $\pi\colon G\to\Bc(\Hc)$ is a unitary irreducible representation of a nilpotent Lie group on a complex Hilbert space, then it is a classical result that for its corresponding Banach algebra representation $\pi\colon L^1(G)\to\Bc(\Hc)$ 
one has $\pi(L^1(G))\subseteq\Kc(\Hc)$, hence $\overline{\pi(L^1(G))}=\Kc(\Hc)$ by the above discussion. 
(See for instance 
 \cite[Cor. 3]{Di59}
and \cite[Th. 4.2.1]{CG90}.)
The main theorem of the present paper (Theorem~\ref{Banach5}) says that the classical result stated above still holds true with the Hilbert space~$\Hc$ replaced by any reflexive Banach space $\Zc$, the unitary irreducible representation~$\pi$ replaced by any irreducible representation satisfying $\sup\limits_{g\in G}\Vert\pi(g)\Vert<\infty$, 
and with~$\Kc(\Hc)$ replaced by the norm-closure of finite-rank operators on $\Zc$. 
We also prove that this conclusion is sharp in the sense that it fails for many representations on non-reflexive Banach spaces (Theorem~\ref{22feb2018}). 
See also \cite{Ne10} for other interesting results on differentiability of vectors with respect to Banach space representations.

The method for obtaining our results consists in 
the study of the way the classical result transfers from Hilbert spaces to Banach space. 
The techniques we  use  rely on intertwining operators between spaces of smooth vectors for 
representations of $G$ on various Banach spaces, and thereby transferring information between these representations. 
In a sequel to the present paper, we will use this method of transference, along with some of the results established here, in order to study multipliers 
associated to unitary irreducible representations of nilpotent Lie groups. 
We note that other interesting applications of Banach space representations of locally compact groups were recently obtained. 
See for instance \cite{DkJW18}, \cite{GT15}, and the references therein.

\subsection*{Notation}
Throughout this paper we denote by $\Sc(\Vc)$ the Schwartz space 
on a finite-dimensional real vector space~$\Vc$. 
That is, $\Sc(\Vc)$ is the set of all smooth functions 
that decay faster than any polynomial together with 
their partial derivatives of arbitrary order. 
We use $\langle\cdot,\cdot\rangle$ to denote any duality pairing between 
finite-dimensional real vector spaces whose meaning is clear 
from the context. 

We also use the convention that Lie groups are denoted by 
upper case Latin letters and Lie algebras are denoted 
by their corresponding lower case Gothic letters. 
For any finite-dimensional real Lie algebra $\gg$ we denote by $\gg_{\CC}:=\CC\otimes_{\RR}\gg$ 
its complexification, which is a complex Lie algebra, and by 
$\U(\gg_{\CC})$ the universal enveloping algebra of $\gg_{\CC}$, 
which is a complex unital associative $*$-algebra 
and is isomorphic to the algebra of distributions with the support at the unit element $\1\in G$ 
for any Lie group $G$ whose Lie algebra is isomorphic to $\gg$. 

By nilpotent Lie group we always mean a connected, simply connected, nilpotent Lie group. 
For such a group $G$, 
its exponential map $\exp_G\colon\gg\to G$ is a diffeomorphism 
whose inverse is 
denoted by $\log_G\colon G\to\gg$. 
Using this diffeomorphism, one defines the Schwartz space  $\Sc(G):=\{\varphi\circ\log_G\mid \varphi\in\Sc(\gg)\}$. 
We will also need $\Cc_0^\infty(G)$, the space of compactly supported 
smooth functions on $G$.

\section
{CCR property on reflexive Banach spaces}
\label{BSRLieGroups}


A uniformly bounded representation of a locally compact group $G$ on a Banach space is said to have the \emph{CCR property} 
if the closure of the image of $L^1(G)$ under that representation is equal to the norm-closure of finite-rank operators on that Banach space. 
This terminology CCR (completely continuous representation) is inspired by the one 
that has been long used and applied for unitary irreducible representations on Hilbert spaces, where the alternative French name `liminaire' was also used. 
See for instance \cite{Fe62}. 

The main results of this section  
concern nilpotent Lie groups and we actually point out in Remark~\ref{GCR} that they cannot be directly extended further.  
Nevertheless, some auxiliary results and lemmas hold true for more general classes of groups 
hence we have stated and proved them in their natural level of generality 
as they may hold an independent interest.


Let $\Xc$ be any complex Banach space with a fixed norm that defines its topology. 
Assume that $\pi\colon G\to\Bc(\Xc)$ is a continuous representation which is \emph{uniformly bounded}, 
in the sense that $\sup\{\Vert\pi(g)\Vert\mid g\in G\}<\infty$. 
Continuity of $\pi$ means that its corresponding map 
$G\times \Xc\to\Xc$, $(g,x)\mapsto\pi(g)x$, is continuous. 
Then one can define the continuous homomorphism of Banach algebras
$$
\pi\colon L^1(G)\to\Bc(\Xc),\quad \pi(\varphi)=\int\limits_G\varphi(g)\pi(g)\de g
$$
where the above integral is strongly convergent.

In the next definition we recall the smooth vectors for Banach space representations, and the smooth operators for Hilbert space representations.

\begin{definition}\label{smooth_oper}
\hfill
\normalfont
\begin{enumerate}
\item\label{smooth_oper_item1} 
The  space of \emph{smooth vectors} for 
the representation $\pi$ is 
$$
\Xc_\infty:=\{x\in\Xc\mid\pi(\cdot)x\in\Ci(G,\Xc)\}.
$$
The linear space $\Xc_\infty$ is endowed with the linear topology 
which makes the linear injective map 
$\Xc_\infty\to\Ci(G,\Xc)$, $x\mapsto\pi(\cdot)x$
into a linear topological isomorphism onto its image. 
Then $\Xc_\infty$ is a Fr\'echet space which is continuously and densely 
embedded in $\Xc$ 
(see for instance \cite[Prop. 2.2]{BB10} and the references therein). 
The \emph{derived representation} 
$$
\de\pi\colon\gg\to
\End(\Xc_\infty),\quad\de\pi(X)x
=\frac{\de}{\de t}\Big\vert_{t=0}\pi(\exp_G(tX))x
$$
is a Lie algebra representation, so it extends to a unital homomorphism of complex associative algebras 
$\de\pi\colon\U(\gg_{\CC})\to \End(\Xc_\infty)$. 
\item\label{smooth_oper_item2} 
Assume that  $\Xc$ is a Hilbert space.
Then the set $\Bc(\Xc)_\infty$ of \emph{smooth operators} for the representation~$\pi$ 
is 
$$
\Bc(\Xc)_\infty:=\{T\in\Bc(\Xc)\mid T(\Xc)\subseteq\Xc_\infty
\text{ and }T^*(\Xc)\subseteq \Xc_\infty\}.
$$ 
(See \cite[Sect. 3]{BB10} for more details and \cite[Cor. 3.1]{BB10} for the natural topology on $\Bc(\Xc)_\infty$ when $\pi$ is a unitary irreducible representation.)
\end{enumerate}

It is easily checked that $\pi(\Sc(G))\Xc\subseteq\Xc_\infty$. 
In particular, one has a representation of the convolution algebra $\Sc(G)$, 
$$
\pi_{\Sc}\colon \Sc(G)\to\End(\Xc_\infty), \quad \pi_{\Sc}(\varphi):=\pi(\varphi)\vert_{\Xc_\infty}.
$$
\end{definition}

The above representation $\pi_{\Sc}$ is used in this paper as a purely algebraic object (a morphism of associative algebras) without any continuity property. 
In this connection we also make the following definition. 

\begin{definition}
\normalfont
With the above notation, 
the uniformly bounded continuous representation $\pi\colon G\to\Bc(\Xc)$ is called \emph{strongly irreducible} if its associated representation 
$\pi_{\Sc}\colon \Sc(G)\to\End(\Xc_\infty)$ is algebraically irreducible, 
that is, the only linear subspaces of $\Xc_\infty$ that are invariant to all the operators in $\pi_{\Sc}(\Sc(G))$ 
are $\{0\}$ and $\Xc_\infty$. 

On the other hand, the representation $\pi$ is called \emph{topologically irreducible} (for short \emph{irreducible}) if 
the only closed linear subspaces of $\Xc$ that are invariant to all the operators in $\pi(G)$ 
are $\{0\}$ and $\Xc$. 
\end{definition}

\begin{example}\label{Banach2}
\normalfont
If $\Hc$ is a Hilbert space and the bounded continuous representation $\pi\colon G\to\Bc(\Hc)$ is topologically irreducible, 
then $\pi$ is also strongly irreducible. 

In fact, since the group $G$ is amenable and $\pi$ is bounded, it follows that $\pi$ is similar to a unitary representation. 
Then, it is known that unitary irreducible representations of nilpotent Lie groups are strongly irreducible; 
see \cite[Cor. 3.4.1]{How77}. 
\end{example}

\begin{lemma}\label{Banach2.1}
Let $G$ be any connected Lie group 
and assume that $\pi\colon G\to\Bc(\Xc)$ and $\tau\colon G\to\Bc(\Yc)$ 
are uniformly bounded 
continuous representations. 
For any topological linear isomorphism $A\colon\Xc_\infty\to\Yc_\infty$ 
the following assertions are equivalent: 
\begin{enumerate}[(i)]
\item For all $g\in G$ one has $A\pi(g)\vert_{\Xc_\infty}=\tau(g)A$. 
\item For all $X\in\gg$ one has $A\de\pi(X)=\de\tau(X)A$. 
\end{enumerate}
\end{lemma}

\begin{proof}
This is a special case of \cite[Cor. 3.3]{Po72}. 
\end{proof}

The following notion of almost equivalence of group representations is suggested by \cite[Def. 3.3]{Po72}. 

\begin{definition}\label{almost-equiv}
\normalfont
In the setting of Lemma~\ref{Banach2.1}, we say the representations $\pi$ and $\tau$ are \emph{almost equivalent}, 
and the operator $A$ is called an \emph{almost equivalence}. 
\end{definition}

\begin{remark}
\normalfont 
The relation of almost equivalence of group representations is reflexive, symmetric, and transitive, 
and for any pair of unitary representations this relation coincides with unitary equivalence by \cite[Th. 3.4]{Po72}.  
\end{remark}

It is interesting to study how various properties of Lie group representations are transferred by almost equivalence.
 Theorem~\ref{Banach5} below  shows that, in the case of nilpotent Lie groups,  
the CCR property of their unitary irreducible representations 
propagates under a suitable form to all their uniformly bounded irreducible representations on reflexive Banach spaces. 


\begin{lemma}\label{Banach2.2}
In the setting of Lemma~\ref{Banach2.1} the following assertions hold: 
\begin{enumerate}[{\rm(i)}]
\item\label{Banach2.2_item1} 
If $A$ is regarded as a densely-defined linear operator from $\Xc$ into $\Yc$, 
then its closure $\overline{A}\colon \Dc(\overline{A})\to\Yc$ is an injective operator. 
\item\label{Banach2.2_item2} 
For every $\varphi\in\Cc_0^\infty(G)$ and $x\in\Dc(\overline{A})$ 
one has $\pi(\varphi)x\in\Xc_\infty\subseteq\Dc(\overline{A})$ and $A\pi(\varphi)x=\overline{A}\pi(\varphi)x=\tau(\varphi)\overline{A}x$. 
\item\label{Banach2.2_item3} If $G$ is a nilpotent Lie group, then the above assertion \eqref{Banach2.2_item2} holds for $\varphi\in\Sc(G)$. 
\end{enumerate}
\end{lemma}

\begin{proof}
That $A$ is a closable operator follows by \cite[Th. 3.2]{Po72}, 
while injectivity of $\overline{A}$ was noted in the proof of \cite[Cor. 3.3]{Po72}. 

Assertions \eqref{Banach2.2_item2}--\eqref{Banach2.2_item3} follow by Hille's theorem which says that closed linear operators commute with the Bochner integral (see \cite[Th. 3.7.12]{HiPh74}). 
\end{proof}

\begin{lemma}\label{Banach3}
Let $G$ be any locally compact group with a set $\Ac\subseteq L^1(G)$ 
that is invariant under right translations by some dense subset $\Gamma$ 
of $G$ and 
such that for every neighborhood $V$ of $\1\in G$ there exists $\varphi\in\Ac$ with $\supp\varphi\subseteq V$, 
$0\le\varphi$ almost everywhere on $G$, and $\int\limits_G\varphi(g)\de g=1$. 

If $\pi\colon G\to\Bc(\Xc)$ is any bounded continuous representation with its corresponding 
Banach algebra representation $\pi\colon L^1(G)\to \Bc(\Xc)$, 
then for any closed linear subspace $\Yc\subseteq\Xc$ one has 
$\pi(G)\Yc\subseteq\Yc\iff\pi(\Ac)\Yc\subseteq\Yc$. 
\end{lemma}

\begin{proof}
The implication ``$\Rightarrow$'' is clear. 
For the converse implication we use the well known fact that for $V$ and  
$\varphi\in\Ac$ as in the statement 
one has 
$$
\Vert\pi(\varphi)x-x\Vert\le\int\limits_G\varphi(g)\Vert\pi(g)x-x\Vert\de g
\le \sup\limits_{g\in V}\Vert\pi(g)x-x\Vert
$$
for all $x\in \Xc$. 
Thus, by using in addition the continuity of $\pi\colon G\to\Bc(\Xc)$, it follows that 
$$
(\forall x\in\Xc)\quad \inf\limits_{\varphi\in\Ac}\Vert\pi(\varphi)x-x\Vert=0.
$$
Then, by the hypothesis that $\Ac$ is invariant under right translations by 
$\Gamma\subseteq G$, 
we obtain 
$$
(\forall x\in\Xc)(\forall g\in\Gamma)\quad 
\inf\limits_{\varphi_{g^{-1}}\in\Ac}\Vert\pi(\varphi)x-\pi(g)x\Vert
=\inf\limits_{\varphi\in\Ac}\Vert\pi(\varphi)\pi(g)x-\pi(g)x\Vert=0,
$$
where $\varphi_{g^{-1}}:=\varphi(\cdot g^{-1})$. 
This implies that if $\Yc\subseteq\Xc$ is any closed linear subspace with 
$\pi(\Ac)\Yc\subseteq\Yc$, then $\pi(\Gamma)\Yc\subseteq\Yc$, 
hence also $\pi(G)\Yc\subseteq\Yc$. 
\end{proof}

We recall that the unitary dual of a Lie group is a set whose points are the unitary equivalence classes of unitary irreducible representations of that group. 

\begin{theorem}\label{Banach4}
Let $\pi\colon G\to\Bc(\Xc)$ be any uniformly bounded continuous representation of a nilpotent Lie group $G$. 
Then $\pi$ is topologically irreducible if and only if it is strongly irreducible. 
If this is the case, then there exists a unique point in the unitary dual of $G$ 
consisting of unitary irreducible representations of $G$ which are almost equivalent to $\pi$ 
and whose spaces of smooth vectors are isomorphic as $\Sc(G)$-modules with the space of smooth vectors of~$\pi$. 
\end{theorem}

\begin{proof}
First let $\pi\colon G\to\Bc(\Xc)$ be any strongly irreducible representation 
and $\{0\}\ne\Yc\subseteq\Xc$ be any closed linear subspace with $\pi(G)\Yc\subseteq\Yc$. 
Then one has also $\pi(\Sc(G))\Yc\subseteq\Yc$ by Lemma \ref{Banach3}. 
On the other hand $\pi(\Sc(G))\Yc\subseteq\Xc_\infty$ and $\pi(\Sc(G))\Yc$ is a dense linear subspace of $\Yc$. 
Moreover, since $\{0\}\ne\Yc$, it easily checked that $\{0\}\ne\pi(\Sc(G))\Yc$ 
hence $\pi(\Sc(G))\Yc=\Xc_\infty$ by the assumption that $\pi$ is strongly irreducible 
along with the fact that the linear space $\pi(\Sc(G))\Yc$ is invariant under the algebra representation 
$\pi_{\Sc}\colon \Sc(G)\to\End(\Xc_\infty)$.
Thus $\Xc_\infty=\pi(\Sc(G))\Yc\subseteq\Yc$, which implies $\Xc=\Yc$, hence $\pi$ is topologically irreducible.

For the converse implication assume that $\pi$ is topologically irreducible. 
Using Lemma~\ref{Banach3} for $\Ac=\Sc(G)$, it then follows that 
$\pi\colon\Sc(G)\to\Bc(\Xc)$ is a topologically irreducible representation, 
and now the conclusion follows 
by \cite[Th.]{Lu90}. 
The uniqueness of the point in the unitary dual of $G$ as in the statement 
follows by \cite[Th. 3.4]{Po72}. 
\end{proof}

The second part of the above theorem says that one can find a unitary irreducible representation 
$\pi_0\colon G\to\Bc(\Hc)$ (on a Hilbert space)
and a closed injective linear operator $A\colon\Dc(A)\to\Xc$ whose domain $\Dc(A)$ is dense in $\Hc$ and contains $\Hc_\infty$,  
with range $\Ran A$ dense in $\Xc$ and contains $\Xc_\infty$, 
and which defines by restriction a topological isomorphism of $\Sc(G)$-modules 
$A\vert_{\Hc_\infty}\colon\Hc_\infty\to\Xc_\infty$. 
Next, we take advantage of the properties of $\pi_0$ to establish Theorem \ref{Banach5}. 
We recall that $\Kc(\Xc)$  denotes the norm-closure of the finite-rank operators in $\Bc(\Xc)$, 
which is equal to the set of all compact operators on~$\Xc$ if the  Banach space $\Xc$ has the approximation property. 

\begin{theorem}\label{Banach5}
Let $\pi\colon G\to\Bc(\Zc)$ be any uniformly bounded continuous representation of a nilpotent Lie group $G$. 
If $\pi$ is topologically irreducible then the following assertions hold: 
\begin{enumerate}[{\rm(i)}]
	\item\label{Banach5_item1} The set $\pi(\Sc(G))$ consists of nuclear operators on~$\Zc$, and 
	one has $\pi(L^1(G))\subseteq \Kc(\Zc)$ 
	\item\label{Banach5_item2}  
	If moreover the Banach space $\Zc$ is reflexive, 
	then $\pi(\Sc(G))$ contains a norm-dense subspace of $\Kc(\Zc)$ and the representation~$\pi$ has the CCR property.
\end{enumerate}
\end{theorem}

\begin{proof} 
\eqref{Banach5_item1} 
We prove first that for every $\varphi\in \Sc(G)$,  the operator $\pi(\varphi)\colon\Zc\to\Zc$ is nuclear. 
To this end note that $\pi(\varphi)\Zc\subseteq\Zc_\infty$, and then $\pi(\varphi)\colon\Zc\to\Zc_\infty$ 
is a continuous operator by the closed graph theorem, using the fact that the inclusion map $\Zc_\infty\hookrightarrow\Zc$ 
is continuous and $\Zc_\infty$ is a Fr\'echet space. 
Then the operator $\pi(\varphi)$ factorizes as 
$\Zc\mathop{\longrightarrow}
\limits^{\pi(\varphi)}\Zc_\infty\hookrightarrow\Zc$.  
Therefore the assertion will follow by 
\cite[Prop. 47.1 and Th. 50.1((a),(c))]{Tr67}
as soon as we will have proved that $\Zc_\infty$ is a nuclear space. 
 
For that, we use the fact that Theorem~\ref{Banach4} gives  
a topological isomorphism of $\Sc(G)$-modules $\Hc_\infty\simeq\Zc_\infty$
The space of smooth vectors $\Hc_\infty$ is well-known to be a nuclear Fr\'echet space 
(see for instance \cite[Cor. 3.1(1.)]{BB10}), hence also $\Zc_\infty$ is a nuclear space.

It follows  that for arbitrary $\varphi\in \Sc(G)$ one has 
$\pi(\varphi)\in \Kc(\Zc)$, and this, together with the continuity of $\pi$, implies the inclusion $\pi(L^1(G))\subseteq\Kc(\Zc)$.

\eqref{Banach5_item2} 
Assume that the Banach space~$\Zc$ is reflexive. 
To prove that the inclusion $\pi(L^1(G))\subseteq\Kc(\Zc)$ is actually an equality, 
it suffices to check that $\pi(\Sc(G))$ contains a dense subspace of $\Kc(\Zc)$. 
More precisely, for a suitable dense linear subspace $\Yc\subseteq\Zc^*$ we  prove that 
\begin{equation}\label{Banach5_proof_eq1}
\Yc\otimes\Zc_\infty\subseteq\pi(\Sc(G)).
\end{equation} 
The set of finite-rank operators 
$$
\overline{\Hc_\infty}\otimes\Hc_\infty:=
\spa\{(\cdot\mid h)k\mid h,k\in\Hc_\infty\}
$$ 
is contained in $\Bc(\Hc)_\infty$, 
where $(\cdot\mid\cdot)$ is the scalar product of the Hilbert space $\Hc$. 
Since the  map 
$\pi_0\colon\Sc(G)\to\Bc(\Hc)_\infty$ is surjective (see \cite[Th. 3.4]{How77}), 
for any $h,k\in\Hc_\infty$ 
there exists $\psi\in\Sc(G)$ such that one has $\pi_0(\psi)v=(v\mid h)k$ for all $v\in\Hc_\infty$.
Using the topological isomorphism of $\Sc(G)$-modules 
$A\colon\Hc_\infty\to\Zc_\infty$ 
we now obtain 
$$
\pi(\psi)Av=A\pi_0(\psi)v=(v\mid h)Ak=(A^{-1}Av\mid h)Ak
=\langle (A^{-1})^*\overline{h},Av\rangle Ak
$$
where $\overline{h}:=(\cdot\mid h)\in\Hc^*\subseteq(\Hc_\infty)^*$ 
and $(A^{-1})^*\colon (\Hc_\infty)^*\to(\Zc_\infty)^*$ 
is the transpose mapping of $A^{-1}$. 
It follows that
\begin{equation}\label{Banach5_proof_eq2}
(\forall w\in\Zc_\infty)\quad \pi(\psi)w
=\langle (A^{-1})^*\overline{h},w\rangle Ak.
\end{equation}
Here we know that $\pi(\psi)\in\Bc(\Zc)$ since $\psi\in\Sc(G)$ and $\pi$ is uniformly bounded, 
hence using also the fact that $\Zc_\infty$ is dense in $\Zc$, 
it follows that for arbitrary $h\in\Hc_\infty$ the functional 
$(A^{-1})^*\overline{h}\in(\Zc_\infty)^*$ is continuous with respect to the norm of $\Zc$ 
and then it uniquely extends to a functional 
$(A^{-1})^*\overline{h}\in\Zc^*$.  
Thus, denoting 
$$
\Yc:=\spa\{(A^{-1})^*\overline{h}\mid h\in\Hc_\infty\}\subseteq\Zc^*
$$
one obtains \eqref{Banach5_proof_eq1}. 
It remains to check that $\Yc$ is a dense subspace of $\Zc^*$. 

To this end we argue by contradiction. 
If $\Yc$ were not dense in $\Zc^*$, then by the Hahn-Banach theorem 
there would exist  
$\gamma\in(\Zc^*)^*\setminus\{0\}$ with $\gamma(y)=0$ for all $y\in\Yc$. 
Using the fact that $\Zc$ is a reflexive Banach space 
it then follows that there exists 
$x\in\Zc\setminus\{0\}$ 
with $\langle (A^{-1})^*\overline{h},x\rangle=0$ for all $h\in\Hc_\infty$.  
By \eqref{Banach5_proof_eq2}, this implies that 
$\pi(\psi)x=0$ for every $\psi\in\Sc(G)$ with 
$\pi_0(\psi)\in (\Hc^*)_\infty\otimes\Hc_\infty$. 

But the surjective map $\pi_0\colon\Sc(G)\to\Bc(\Hc)_\infty$ is open by the open mapping theorem, 
and on the other hand the set of finite-rank operators $\overline{\Hc_\infty}\otimes\Hc_\infty$  
is dense in the Fr\'echet space $\Bc(\Hc)_\infty$ 
by \cite[Cor. 3.3]{BB10}.
Then it is straightforward to check that 
$\{\psi\in\Sc(G)\mid \pi_0(\psi)\in \overline{\Hc_\infty}\otimes\Hc_\infty\}$ 
is a dense subset of $\Sc(G)$. 
This implies $\pi(\psi)x=0$ for every $\psi\in\Sc(G)$, hence $x=0$, 
which is a contradiction with $x\in\Zc\setminus\{0\}$. 
This completes the proof of the theorem. 
\end{proof}

In order to point out the significance of Theorem~\ref{Banach5}, 
we recall that many important Banach spaces, 
as for instance $L^p$-spaces or more general mixed-norm Lebesgue spaces $L^{p_1,\dots,p_k}$, carry 
irreducible representations of nilpotent Lie groups or even exponential solvable Lie groups 
(that is, Lie groups $G$ whose exponential map $\exp_G\colon\gg\to G$ is bijective); 
see \cite{LuMo01} and \cite{LuMiMo03}. 
Let us also make the following remark that shows 
that Theorem~\ref{Banach5} cannot be extended beyond nilpotent Lie groups. 

\begin{remark}\label{GCR}
\normalfont
Let $G$ be any exponential solvable Lie group.
Then the following assertions are equivalent:
\begin{enumerate}[(i)]
\item\label{GCR_item1} $G$ is a  CCR (liminary) 
group, that is, every unitary irreducible 
representation of $G$ on a Hilbert space has the CCR property; 
\item\label{GCR_item2} $G$ is type R, that is, for every $g\in G$ all the eigenvalues
of the operator $\Ad_G g\colon\gg_{\CC}\to\gg_{\CC}$ belong to  the
unit circle $\TT$
(equivalently, for every $x\in\gg$, all the eigenvalues of the
operator $\ad_{\gg} x\colon\gg_{\CC}\to\gg_{\CC}$
are purely imaginary);
\item\label{GCR_item3} every coadjoint orbit of $G$ is a closed subset of $\gg^*$;
\item\label{GCR_item4} $G$ is a nilpotent Lie group.
\end{enumerate}
In fact, since $G$ is an exponential Lie group, 
it is a connected, simply connected, solvable Lie group of type~I 
(see \cite[Sect. 0, Rem. 3]{AuKo71}).
Then, using also Glimm's characterization of separable $C^*$-algebras of type~I,
we obtain that $G$ has the property GCR (is postliminary),  
that is, for every unitary irreducible representation $\pi\colon G\to\Bc(\Hc_\pi)$ all compact operators on $\Hc_\pi$  are contained in the norm-closure of $\pi(L^1(G))$.
Hence \eqref{GCR_item1}--\eqref{GCR_item2} are equivalent by \cite[Ch. V, Th. 1--2]{AuMo66}.

By \cite[Prop. 5.2.13, Th. 5.2.16]{FuLu15}, since $G$ is an
exponential Lie group,
it follows that for every $x\in\gg$, the operator 
$\ad_G x\colon\gg_{\CC}\to\gg_{\CC}$
has no nonzero purely imaginary eigenvalues,
hence 
\eqref{GCR_item2} and \eqref{GCR_item4} 
 are equivalent.
Finally, the equivalence 
of \eqref{GCR_item1} to \eqref{GCR_item3} is
discussed in \cite[Th. 5.3.31 infra]{FuLu15} (see also \cite[Th.
1]{Pu68}).
\end{remark}

\section{CCR fails on non-reflexive Banach spaces}


We now turn to a obtaining a result (Theorem~\ref{22feb2018}) which shows that the CCR property in Theorem~\ref{Banach5} may not be obtained if the hypothesis of reflexivity of the Banach space $\Zc$ in removed, 
and also that the CCR property of Banach space representations is not preserved by almost equivalence. 

\begin{lemma}\label{21feb2018}
Let $\Xc$ be a Banach space with a closed linear subspace $\Theta\subsetneqq\Xc^*$,  
and denote by $\Kc_\Theta(\Xc)$ the norm-closed linear subspace of $\Kc(\Xc)$ spanned by the rank-one operators $\theta(\cdot)y$ with $\theta\in\Theta$ and $y\in\Xc$. 
Then one has $\Kc_\Theta(\Xc)\subsetneqq\Kc(\Xc)$. 
\end{lemma}

\begin{proof}
Since $\Theta\subsetneqq\Xc^*$, there exists $\xi\in\Xc^*$ with \begin{equation}\label{21feb2018_eq1}
(\forall\theta\in\Theta)\quad \Vert\xi-\theta\Vert\ge 2. 
\end{equation}
We now fix $x\in\Xc$ arbitrary with $\Vert x\Vert=1$ and we will check that 
\begin{equation}\label{21feb2018_eq2}
(\forall K\in\Kc_\Theta(\Xc))\quad \Vert\xi(\cdot)x- K\Vert\ge 1/2, 
\end{equation}
and in particular the rank-one operator $\xi(\cdot)x$ does not belong to $\Kc_\Theta(\Xc)$, which implies the assertion. 

It suffices to prove \eqref{21feb2018_eq2} for operators $K$ for which there exist an integer $n\ge 1$, vectors $y_1,\dots,y_n\in\Xc$ and functionals $\theta_1,\dots,\theta_n\in\Theta$ with 
$$K=\sum\limits_{j=1}^n\theta_j(\cdot)y_j, $$
 since the set of all operators of this form is norm-dense in $\Kc_\Theta(\Xc)$. 
To estimate the left-hand side of \eqref{21feb2018_eq2} we recall that 
for any operator $T\in\Bc(\Xc)$ one has 
\begin{equation}\label{21feb2018_eq3}
\Vert T\Vert=\sup\{\vert\eta(Tv) \vert\mid v\in\Xc, \Vert v\Vert=1; \ \eta\in\Xc^*,\Vert\eta\Vert=1\}, 
\end{equation}
as a direct consequence of the Hahn-Banach theorem.

Since $\Vert x\Vert=1$,  again by the Hahn-Banach theorem,  there is $\eta\in \Xc^*$ with $\Vert\eta\Vert=1$ 
and 
\begin{equation}\label{21feb2018_eq4}
\vert \eta(x)\vert>1/2.
\end{equation}
In particular, $\eta (x)\ne0$. 

For all $t_1,\dots,t_n\in\CC$ one has $\Vert\xi-\sum\limits_{j=1}^n t_j\theta_j\Vert\ge 2$ by \eqref{21feb2018_eq1}, hence 
there exists $v\in\Xc$ with $\Vert v\Vert=1$ and 
$\vert\xi(v)-\sum\limits_{j=1}^n t_j\theta_j(v)\vert\ge 1$. 
Using this with $t_j:=\eta(y_j)/\eta(x)$ for $j=1,\dots,n$, we obtain 
\begin{equation}\label{21feb2018_eq5}
(\exists v \in\Xc,\ \Vert v \Vert=1)\quad 
\vert\eta(x)\xi(v)-\sum\limits_{j=1}^n \eta(y_j)\theta_j(v)\vert\ge \vert\eta (x)\vert. 
\end{equation}
Now, denoting $T:=\xi(\cdot)x- K=\xi(\cdot)x- \sum\limits_{j=1}^n\theta_j(\cdot)y_j$, one has
$$
\eta (Tv )
=\eta (x)\xi(v)-\sum\limits_{j=1}^n \eta(y_j)\theta_j(v)
$$
and, by \eqref{21feb2018_eq4}--\eqref{21feb2018_eq5}, 
we obtain $\vert\eta(Tv)\vert\ge 1/2$. 
Then $\Vert T\Vert\ge 1/2$ by \eqref{21feb2018_eq3}, 
which completes the proof of \eqref{21feb2018_eq2}. 
\end{proof}

\begin{lemma}\label{non-dense}
Let $\Vc$ be a finite-dimensional real vector space with $\Vc\ne\{0\}$, and define the Banach space $\Xc:=L^1(\Vc)$. 
Let $\Kc_1(\Xc)$ be the norm-closed linear subspace of $\Bc(\Xc)$ 
spanned by the integral operators on $\Xc$ defined by integral kernels in $\Sc(\Vc\times\Vc)$. Then one has $\Kc_1(\Xc)\subsetneqq\Kc(\Xc)$.
\end{lemma}

\begin{proof}
For every $K\in\Sc(\Vc\times\Vc)$ let us denote by $T_K\in\Bc(\Xc)$ the integral operator defined by the integral kernel~$K$. 
We also define the mapping 
$$\Psi\colon \Sc(\Vc\times\Vc)\to\Bc(\Xc),\quad 
\Psi(K):=T_K.$$
One has that
$$(\forall K\in \Sc(\Vc\times\Vc))\quad 
\Vert T_K\Vert\le\sup_{y\in\Vc}\int\limits_{\Vc}\vert K(x,y)\vert\de x, $$
thus the mapping $\Psi$ is continuous with respect to the usual Fr\'echet topology of $\Sc(\Vc\times\Vc)$. 
Since the algebraic tensor product $\Sc(\Vc)\otimes\Sc(\Vc)$ is dense in $\Sc(\Vc\times\Vc)$, it then follows that 
$\Kc_1(\Xc)$ is equal to the norm-closure of the set 
$\{\Psi(K)\mid K\in \Sc(\Vc)\otimes\Sc(\Vc)\}$. 
As this set consists of finite-rank operators on $\Xc$, 
we obtain $\Kc_1(\Xc)\subseteq\Kc(\Xc)$. 

To check that this inclusion is strict we define $\Theta$ as the norm-closure of $\Sc(\Vc)$ viewed as a linear subspace of the Banach space $\Xc^*\simeq L^\infty(\Vc)$. 
That is, $\Theta$ is the space of continuous functions on~$\Vc$ that vanish at~$\infty$. 
One clearly has $\Theta\subsetneqq \Xc^*$, hence $\Kc_\Theta(\Xc)\subsetneqq\Kc(\Xc)$ by Lemma~\ref{21feb2018}. 
On the other hand, since $\Sc(\Vc)\subseteq\Theta$, it is readily seen that $\Kc_1(\Xc)\subseteq\Kc_\Theta(\Xc)$, hence $\Kc_1(\Xc)\subsetneqq\Kc(\Xc)$. 
\end{proof}

\begin{theorem}\label{22feb2018} 
If $G$ is a nilpotent Lie group, then for every unitary irreducible representation $\pi\colon G\to\Bc(\Hc)$ on a Hilbert space~$\Hc$ with $\dim\Hc=\infty$ there exists a
uniformly bounded continuous representation $\pi_1\colon G\to\Bc(\Xc)$ 
which is almost equivalent to $\pi$ and does not have the CCR property.
\end{theorem}

\begin{proof}
We recall from \cite[Lemma 2.1]{ER95} that 
one can find a 
finite-dimensional real vector space $\Vc$ 
and a family of continuous representations $\pi_p$ 
of $G$ by isometries of $L^p(\Vc)$ for every $p\in[1,\infty)$, 
satisfying the following conditions: 
\begin{enumerate}
	\item The unitary representations $\pi$ and $\pi_2$ are unitary equivalent (so $\Vc\ne\{0\}$). 
	\item
	There exist continuous functions 
	$P \colon  G\times \Vc\to\RR$ and $Q \colon  G\times \Vc\to \Vc$
	such that, for every $g\in G$,  the function 
	$P(g,\cdot)$ is a polynomial 
	and the map 
	$Q(g,\cdot)$
	is a measure-preserving polynomial diffeomorphism of~$\Vc$ with 
	$(\pi_p(g)\varphi)(v)=\ee^{\ie P(g,v)}\varphi(Q(g,v))$ 
	for all $g\in G$, $v\in\Vc$, $\varphi\in \Sc(\Vc)$, 
	and $p\in[1,\infty)$. 
\end{enumerate}

We now check that the representation $\pi_1$ does not have the CCR property. 
By \cite[Th. 4.2.1]{CG90}, 
for every $p\in[1,\infty)$ one has $\pi_p(\Sc(G))\subseteq\Kc_p(L^p(\Vc))$, 
where $\Kc_p(L^p(\Vc))$ denotes the norm-closed linear subspace spanned by the integral operators on $L^p(\Vc)$ defined by integral kernels in $\Sc(\Vc\times\Vc)$. 
For $p=1$ one obtains 
 $\pi_1(\Sc(G))\subseteq\Kc_1(\Xc)\subsetneqq\Kc(\Xc)$ by Lemma~\ref{non-dense}, which directly implies that the representation $\pi_1$ does not have the CCR property. 

To complete the proof it remains to show that the representations $\pi_2$ and $\pi_1$ are almost equivalent. 
To this end we will check that the space of smooth vectors for 
$\pi_p$ is $\Sc(\Vc)$ for arbitrary $p\in[1,\infty)$. 
We denote by  $\Pc(\Vc)$ the set of all linear partial differential operators on $\Vc$ with polynomial coefficients.
It follows by \cite[Th. 4.1.1]{CG90} that $\de\pi_2(\U(\gg_{\CC}))=\Pc(\Vc)$. 
Since $\pi_p$ and $\pi_2$ agree on $\Sc(\Vc)\subseteq L^p(\Vc)\cap L^2(\Vc)$, we actually have that  $\de\pi_2(\U(\gg_{\CC}))=\de\pi_p(\U(\gg_{\CC}))=\Pc(\Vc)$. 
On the other hand, for $1\le p\le \infty$,  the seminorms on $\Sc(\Vc)$ 
$$ \varphi\mapsto \sup_{|\alpha|+|\beta|\le  k} \Vert x^\alpha D^\beta \varphi\Vert_{L^p(\Vc)}, \quad k\in \NN, $$
are continuous and define the same topology. 
We thus obtain that, if   $\{X_1,\dots,X_m\}$ is a basis in~$\gg$ 
and $1\le p<\infty$, 
the seminorms 
$$ \varphi\mapsto \sup_{ m\le  k} \Vert \de\pi_p(X_{j_1})\cdots\de\pi_p(X_{j_m})\varphi \Vert_{L^p(\Vc)}, \quad k\in \NN, $$
are continuous and define the same topology on $\Sc(\Vc)$. 
Then, by \cite[Corollary~4.1]{Po72}, the spaces of smooth vectors for 
$\pi_p$ and $\pi_2$ are the same, and this space is $\Sc(\Vc)$.  
(See also \cite[page 346]{How77} for the relation between $\Sc(\Vc)$ and $\Pc(\Vc)$.)
This concludes the proof. 
\end{proof}


\begin{remark}\label{transitive1}
	\normalfont
	Assume the setting of Theorem~\ref{Banach5} and denote by $\Ac_\pi$ the norm-closure of $\pi(L^1(G))$. Then   
	the CCR property of $\pi$ means $\Ac_\pi=\Kc(\Zc)$ when $\Zc$ is reflexive. 
	An alternative proof of this equality, which however gives no information on~$\pi(\Sc(G))$, can be obtained as follows. 
	
	The set $\Ac_\pi$ is a closed subalgebra of $\Bc(\Zc)$. 
	Since the representation $\pi$ is irreducible,  
	then Lemma~\ref{Banach3} shows that the operator algebra $\Ac_\pi$ is transitive, that is, 
	there are no nontrivial invariant subspaces $\Ac_\pi$. 
	Moreover, one has $\Ac_\pi\subseteq\Kc(\Zc)$ by Theorem~\ref{Banach5}\eqref{Banach5_item1}, 
	and in particular $\Ac_\pi\cap\Kc(\Zc)\ne\{0\}$. 
	If the Banach space $\Zc$ is reflexive, it then follows by \cite[Cor. 7.4.8]{RaRo00} that $\Ac_\pi$ contains all finite-rank operators on~$\Zc$, hence $\Ac_\pi\supseteq\Kc(\Zc)$, 
	which implies $\Ac_\pi=\Kc(\Zc)$.
	
	On the other hand, it is known from \cite[Ex. 7.4.4]{RaRo00} that on the dual of every non-reflexive Banach space there exists a transitive norm-closed operator algebra that contains some but not all finite-rank operators. 
	The result of our Theorem~\ref{22feb2018} and its proof provide examples of transitive operator algebras of this type which live however on a different class of Banach spaces, namely on Banach spaces of the form $L^1(\Vc)$ where $\Vc$ is a finite-dimensional real vector space. 
\end{remark}

\subsection*{Acknowledgment}
We wish to thank Jean Ludwig and Yuri Turovski\u\i\ for their kind assistance with some pertinent references and remarks.

\end{document}